\documentclass{patmorin}
\usepackage{pat}
\usepackage{amsopn}
\DeclareMathOperator{\obs}{obs}
\newcommand{\x}{\mathsf{x}}
\newcommand{\y}{\mathsf{y}}

\title{\MakeUppercase{On Obstacle Numbers}}
\author{Vida Dujmovi\'c and Pat Morin}

\begin{document}
\begin{titlepage}
\maketitle

\begin{abstract}
\setlength{\baselineskip}{16.8pt}
The obstacle number is a new graph parameter introduced by Alpert, Koch,
and Laison (2010).  Mukkamala \etal\ (2012) show that there exist graphs
with $n$ vertices having obstacle number in $\Omega(n/\log n)$. In this
note, we up this lower bound to $\Omega(n/(\log\log n)^2)$.  Our proof
makes use of an upper bound of Mukkamala \etal\ on the number of graphs
having obstacle number at most $h$ in such a way that any subsequent
improvements on their upper bound will improve our lower bound.
\end{abstract}
\end{titlepage}

\section{Introduction}

\setlength{\baselineskip}{16.8pt}
The obstacle number is a new graph parameter introduced by Alpert, Koch,
and Laison \cite{alpert.koch.ea:obstacle}.  Let $G=(V,E)$ be a graph,
let $\varphi:V\to \R^2$ be a one-to-one mapping of the vertices of
$G$ onto $\R^2$, and let $S$ be a set of connected subsets of $\R^2$.
The pair $(\varphi,S)$ is an \emph{obstacle representation} of $G$ when,
for every pair of vertices $u,w\in V$, the edge $uw$ is in $E$ if and only
if the open line segment with endpoints $\varphi(u)$ and $\varphi(w)$ does
not intersect any \emph{obstacle} in $S$.  An obstacle representation
$(\varphi,S)$ is an \emph{$h$-obstacle} representation if $|S|=h$.
The \emph{obstacle number} of a graph $G$, denoted by $\obs(G)$, is the
minimum value of $h$ such that $G$ has an $h$-obstacle representation.

\figref{fivebyfive} shows a surprising example of a 1-obstacle
representation of the $5\times 5$ grid graph, $G_{5\times 5}$, that
was given to us by Fabrizio Frati. In this figure, the single obstacle
is drawn as a shaded region. Since at least one obstacle is clearly
necessary to represent any graph other than a complete graph, this
proves that $\obs(G_{5\times 5}) = 1$.  (A similar drawing can be used
to show that the $a\times b$, grid graph has obstacle number 1, for any
integers $a,b>1$.)

\begin{figure}[hbpt]
  \begin{center}
    \includegraphics{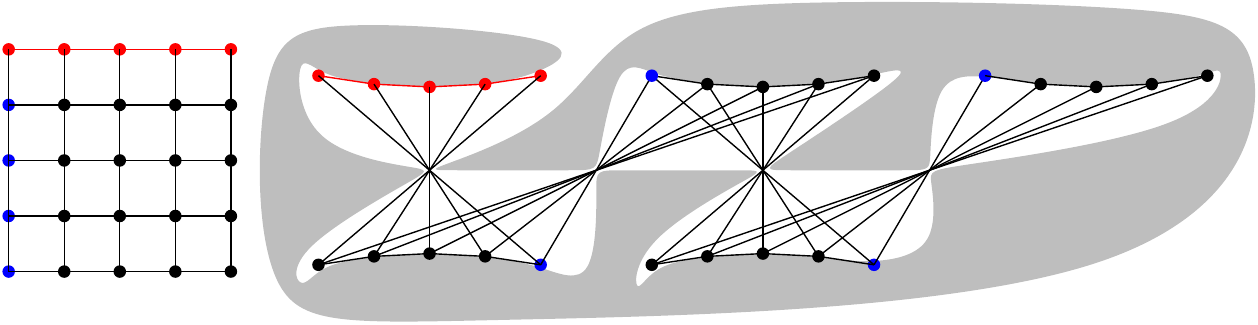}
  \end{center}  
  \caption{The $5\times 5$ grid graph has obstacle number 1.}
  \figlabel{fivebyfive}
\end{figure}

Since their introduction, obstacle numbers have generated significant
research interest 
\cite{%
   fulek.saeedi:convex,%
   johnson.sarioz:computing,%
   mukkamala.pach.ea:lower,%
   mukkamala.pach.ea:graphs,%
   pach.sarioz:small,%
   pach.sarioz:on,%
   sarioz:approximating%
}.
A fundamental---and far from answered---question about obstacle numbers
is that of determining the \emph{worst-case obstacle number},
\[
    w(n) = \max \{\obs(G) :\mbox{$G$ is a graph with $n$ vertices}\}
    \enspace ,
\] 
of a graph with $n$ vertices.

For a graph $G=(V,E)$, we call elements of $\binom{V}{2}\setminus E$
\emph{non-edges} of $G$.  The worst-case obstacle number $w(n)$ is
obviously upper-bounded by $\binom{n}{2}\in O(n^2)$ since, by mapping
the vertices of $G$ onto a point set in sufficiently general position,
one can place a small obstacle---even a single point---on the mid-point
of each non-edge of $G$.  No upper-bound on $w(n)$ that is asymptotically
better than $O(n^2)$ is known.

More is known about lower-bounds on $w(n)$.  Alpert \etal\
initially show that the worst-case obstacle number is
$\Omega(\sqrt{\log n/\log\log n})$ and posed as an open problem the question
of determining if $w(n)\in\Omega(n)$.
Mukkamala \etal\ \cite{mukkamala.pach.ea:graphs} showed that $w(n)\in
\Omega(n/\log^2 n)$ and Mukkamala \etal\ \cite{mukkamala.pach.ea:lower}
later increased this to $w(n)\in\Omega(n/\log n)$.  In the current paper,
we up the lower-bound again by proving the following theorem:
\begin{thm}\thmlabel{main}
  For every integer $n>0$, $w(n)\in\Omega(n/(\log\log n)^2)$, i.e., there
  exist graphs, $G$, with $n$ vertices and $\obs(G)\in\Omega(n/(\log\log
  n)^2)$.
\end{thm}

The proof of \thmref{main} makes use of an upper bound of Mukkamala \etal\
\cite[Theorem~1]{mukkamala.pach.ea:lower} on the number of graphs having
obstacle number at most $h$ in such a way that any subsequent improvements
on their upper bound will result in an improved lower bound on $w(n)$.

\section{The Proof}

Our proof strategy is an application of the probabilistic method
\cite{alon.spencer:probabilistic}.  
We will show that, for a random graph,
$G$, with a fixed embedding, the probability, $p$, that this embedding
allows for an obstacle representation with few obstacles is extremely
small.  We will then show that the number, $N$, of combinatorially
distinct embeddings is not too big.  Small and big in this case are
defined so that $pN < 1$.  Therefore, by the union bound, there exists at
least one graph, $G'$, that has no embedding that allows for an obstacle
representation with few obstacles.  In other words, $\obs(G')$ is large.

\subsection{A Random Graph with a Fixed Embedding}

We make use of the following theorem, due to Mukkamala, Pach, and
P\'alv\"olgyi \cite[Theorem~1]{mukkamala.pach.ea:lower} about the number
of $n$ vertex graphs with obstacle number at most $h$:
\begin{thm}[Mukkamala, Pach, and P\'alv\"olgyi 2012]\thmlabel{pach-gang}
  For any $h\ge 1$, the number of graphs having $n$ vertices and
  obstacle number at most $h$ is at most $2^{O(hn\log^2 n)}$.
\end{thm}

Recall that an Erd\"os-Renyi random graph $G_{n,\frac{1}{2}}$ is a
graph with $n$ vertices and each pair of vertices is chosen as an edge
or non-edge with equal probability and independently of every other pair
of vertices \cite{erdos.renyi:random}.  The following lemma shows that,
for random graphs, a fixed embedding is \emph{very} unlikely to yield
an obstacle representation with few obstacles.

\begin{lem}\lemlabel{fixed}
  Let $G=(V,E)$ be an Erd\"os-R\'enyi random graph $G_{n,\frac{1}{2}}$,
  let $\varphi\from V\to \R^2$ be a one-to-one mapping that is
  independent of the choices of edges in $G$, and let $(\varphi, S)$ be
  an obstacle representation of $G$ using the minimum number of obstacles
  (subject to $\varphi$).  Then, for any constant $c>0$,
  \[
     \Pr\{|S| \in \Omega(n/(\log\log n)^2) \ge 1-e^{-\Omega(cn\log n)}  \enspace .
  \] 
\end{lem}

\begin{proof}
Let $P\subset\R^2$ denote the image of $\varphi$.  Fix some integer $k$ to be
specified later and first consider some arbitrary subset $P'\subset P$
of $k$ points and let $G'=(V',E')$ be the subgraph of $G$ induced by
the set $V'=\{\varphi^{-1}(x):x\in P'\}$ of vertices that are mapped by
$\varphi$ to $P'$.  Applying \thmref{pach-gang} with $n=k$ and $h=\alpha
k/\log^2 k$, we obtain
\begin{equation}
     \Pr\{\obs(G') \le \alpha k/\log^2 k\} 
       \le \frac{2^{O(\alpha k^2)}}{2^{\binom{k}{2}}}
       = e^{-\Omega(k^2)} \enspace , \eqlabel{g1}
\end{equation}
for a sufficiently small constant $\alpha > 0$.  Note that, if
$\obs(G')\ge h$, then, in the obstacle representation $(\varphi,S)$,
there must be at least $h-1$ obstacles of $S$ that are contained in the
convex hull of $P'$.

Without loss of generality assume that no two points in $P$ have the
same x-coordinate and denote the points in $P$ by $x_0,\ldots,x_{n-1}$
by increasing order of x-coordinate.  Let $m=\lfloor n/k\rfloor$ and
consider the point sets $P'_0,\ldots,P'_{m-1}$, where
\[ 
  P_i'=\{x_{ik},x_{ik+1},\ldots,x_{ik+k-1}\} \enspace .
\]  
That is, $P_0',\ldots,P_{m-1}'$ are determined by vertical slabs,
$s_0,\ldots,s_{m-1}$ that each contain $k$ points.  \Eqref{g1} shows
that, with probability at least $1-2^{-\Omega(k^2)}$, the obstacle number
of the subgraph that maps to $P'_i$ is $\Omega(k/\log^2 k)$.  If this
occurs, then $S$ has $\Omega(k/\log^2 k)$ obstacles that are completely
contained in the slab $s_i$.  These obstacles are therefore disjoint
from any other obstacles contained in any other slab $s_j$, $j\neq i$.

We are proving a lower bound on the number of obstacles, so we are
worried about the case where the number of slabs that do \emph{not}
completely contain at least $\alpha k/\log^2 k$ obstacles exceeds $m/2$.
The number of slabs, $M$, not containing at least $\alpha k/\log^2 k$
obstacles is dominated by a binomial$(m,2^{-\Omega(k^2)})$ random
variable.  Using Chernoff's bound on the tail of a binomial random
variable,\footnote{%
  Chernoff's Bound: For any binomial$(m,p)$ random variable, $B$,
  any $\delta>0$ and $\mu=mp$, 
  \[ \Pr\{B\ge (1+\delta)\mu\}
     \le \left(\frac{e^{\delta}}{(1+\delta)^{1+\delta}}\right)^{\mu} 
       \enspace . 
  \]}
we have that
\begin{align*}
  \Pr\{M \ge m/2\} & = \Pr\{M\ge (1+\delta)\mu\}
    & \text{(where $\mu=me^{-ck^2}$ and $\delta=e^{ck^2-1}-1$)} \\
    & \le \left(\frac{e^{\delta}}{(1+\delta)^{1+\delta}}\right)^{\mu} \\
    & = \left(\frac{e^{e^{ck^2}}}{(e^{ck^2-1})^{e^{ck^2-1}}}\right)^{me^{-ck^2}}\\
    & = \left(\frac{e^{e^{ck^2}}}{e^{(ck^2-1)e^{ck^2-1}}}\right)^{me^{-ck^2}}\\
    & = \frac{e^{m}}{e^{m(ck^2-1)e^{ck^2-1}e^{-ck^2}}} \\
    & = \frac{e^{m}}{e^{m(ck^2-1)/e}} \\
    & = e^{-\Omega(mk^2)} \enspace .
\end{align*}
Taking $k=\sqrt{c}\log n$ and recalling that $m=\lfloor n/k\rfloor$, we obtain
the desired result.  In particular,
\[
    |S| \ge \Omega\left(\left(k/\log^2 k\right)\times m \right)
      = \Omega\left(n/(\log\log n)^2\right)
\]
with probability at least
\[
    1-e^{-\Omega(mk^2)} = 1-e^{-\Omega(cn\log n)} \enspace . \qedhere
\]
\end{proof}

We have completed the first step in our application of the probabilistic
method.  \lemref{fixed} shows that the probability, $p$, that a
particular embedding of the random graph $G$ is able to yield an obstacle
representation with $o(n/(\log\log n)^2)$ obstacles is extremely small.
The remaining difficulty is establishing a sufficiently strong upper-bound
on $N$, the number of embeddings of $G$. In actuality, the number of
embeddings is uncountable.  However, we are interested in the number of
``combinatorially distinct'' embeddings.  In particular, we would like
to partition the set of embeddings into equivalence classes such that,
within each equivalence class, the minimum number of obstacles in an
obstacle representation remains the same.

Classifying embeddings (i.e., labelled sets of $n$ points) into
combinatorially distinct equivalence classes has been considered
previously. Several definitions of equivalence exist, including oriented
matroid (a.k.a., chirotope) equivalence
\cite{bland.vergnas:orientability,folkman.lawrence:oriented},
semispace equivalence \cite{goodman.pollack:semispaces}, order equivalence
\cite{goodman.pollack:multidimensional}, and combinatorial equivalence
\cite{goodman:on,goodman.pollack:semispaces}.  For the latter two
definitions of equivalence, the number of distinct (equivalence classes
of) point sets is $e^{O(n\log n)}$ \cite{goodman.pollack:upper}.


Unfortunately, neither order types nor combinatorial-types are
sufficient for answering questions about obstacle representations.
To see this, consider the two embeddings of the same graph shown in
\figref{order-type-problem}.  These two embeddings have the same order
type and the same combinatorial type. However, the embedding on the right
admits an obstacle representation with one obstacle, while the one on
the left requires two obstacles. To see why this is so, observe that each
embedding needs an obstacle on the outer face (shown). For the embedding
on the right, this single obstacle is sufficient, but the embedding
on the left needs an additional obstacle inside one of the inner faces.

\begin{figure}[htbp]
  \begin{center}
    \includegraphics{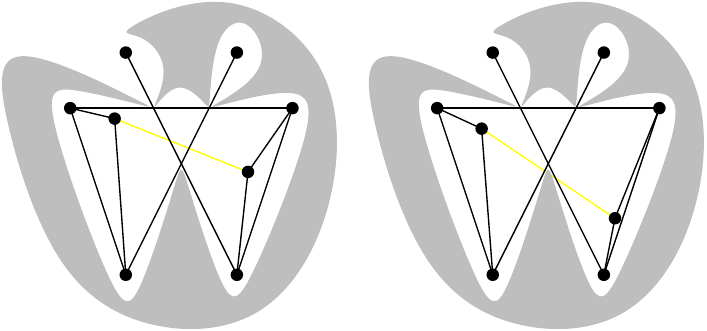}
  \end{center}
  \caption{Order type and combinatorial type are insufficient to determine 
      the number of obstacles needed in an obstacle representation. The yellow
      segment represents a non-edge.}
  \figlabel{order-type-problem}
\end{figure}

\subsection{Super-Order Types}

We now define an equivalence relation on point sets that is
strong enough for our purposes.  Consider a sextuple $T=\langle
a_1,a_2,b_1,b_2,c_1,c_2\rangle$ of points such that
\begin{enumerate}
  \item  $a_1\neq a_2$, $b_1\neq b_2$, $c_1\neq c_2$, 
  \item  $\{a_1,a_2\}\neq \{b_1,b_2\}$, 
$\{b_1,b_2\}\neq \{c_1,c_2\}$, $\{c_1,c_2\}\neq \{a_1,a_2\}$, and 
  \item $\{a_1,a_2\}\cap\{b_1,b_2\}\cap\{c_1,c_2\}=\emptyset$.
\end{enumerate}
We call a sextuple $T$ with this property an \emph{admissible} sextuple.
Let $A$ denote the directed line through $a_1$ and $a_2$ that is directed
from $a_1$ towards $a_2$. Define $B$ and $C$ similarly, but with respect
to $b_1,b_2$ and $c_1,c_2$, respectively.  We say that the sextuple,
$T$, is \emph{degenerate} if
\begin{enumerate}
  \item any of $A$, $B$, or $C$ is vertical;
  \item $A$ is parallel to $B$ or to $C$; or
  \item $A$, $B$, and $C$ contain a common point.
\end{enumerate}
We define the \emph{type}, $\sigma(T)$, of $T$ as
\[
    \sigma(T) = \left\{\begin{array}{rl}
      -1 & \text{if $A\cap B$ comes before $A\cap C$ on $A$.} \\
      0 & \text{if $T$ is degenerate} \\
      +1 & \text{otherwise ($A\cap B$ comes after $A\cap C$ on $A$).} 
    \end{array}\right.
\]
(See \figref{super-order}.)  Let $\langle\langle
i_{1,\ell},i_{2,\ell},j_{1,\ell},j_{2,\ell},k_{1,\ell},k_{2,\ell}\rangle:
\ell \in \{1,\ldots,r\}\rangle$ be any sequence that lists the
admissible sextuples of the index set $\{1,\ldots,n\}$.  Note that $r<
\binom{n}{2}^3$.  The \emph{super-order type} of a sequence $P=\langle
x_1,\ldots,x_n\rangle$ of $n$ distinct points is the sequence
\[
   \sigma(P) = \left\langle \sigma\left(x_{i_{1,\ell}},x_{i_{2,\ell}},
       x_{j_{1,\ell}},x_{j_{2,\ell}},
       x_{k_{1,\ell}},x_{k_{2,\ell}}\right) : \ell\in\left\{1,\ldots,r\right\} \right\rangle \enspace .
\]
Finally, we say that super-order type is \emph{simple} if it contains
no zeros and a sequence of points is \emph{simple} if its super-order
type is simple.  The following lemma shows that super-order types are
sufficient for answering questions about obstacle representations.

\begin{figure}[htbp]
  \begin{center}
    \begin{tabular}{cc}
       \includegraphics{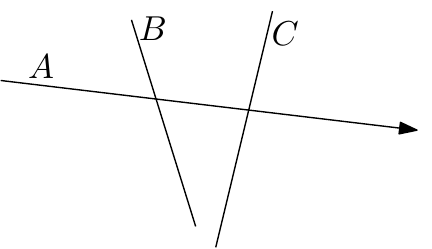} &
       \includegraphics{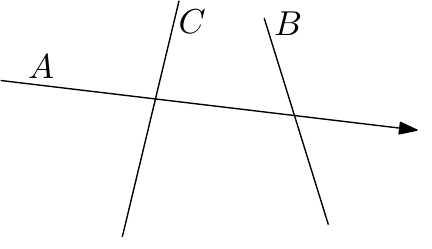} \\ 
       -1 & 1 \\
    \end{tabular}\\[4ex]
    \begin{tabular}{ccc}
       \includegraphics{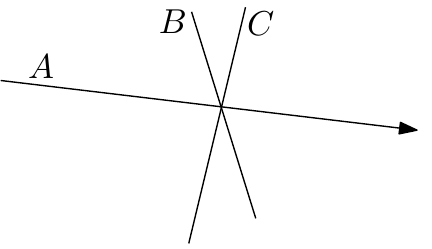} &
       \includegraphics{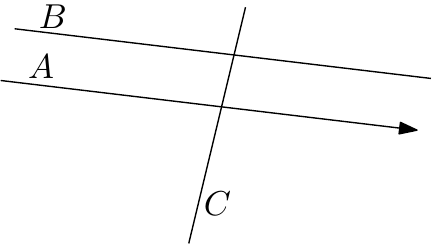} &
       \includegraphics{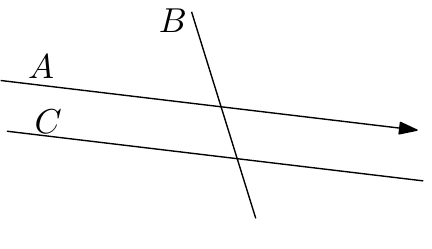} \\
       0 & 0 & 0
    \end{tabular}
  \end{center}
  \caption{The three cases of three lines defined by 6 points that can
      occur in the super-order type.}
  \figlabel{super-order}
\end{figure}

\begin{lem}\lemlabel{order-type}
  Let $G$ be a graph with vertex set $V=\{1,\ldots,n\}$ and let
  $\varphi_1\from V\to\R^2$ and $\varphi_2\from V\to\R^2$ be two embeddings
  of $G$ such that $\langle\varphi_1(1),\ldots,\varphi_1(n)\rangle$
  and $\langle\varphi_2(1),\ldots,\varphi_2(n)\rangle$ have the same
  super-order type.  Then, if $G$ has an $h$-obstacle representation
  $(\varphi_1,S_1)$ then it also has $h$-obstacle representation
  $(\varphi_2,S_2)$.
\end{lem}

\begin{proof}
  Consider the two plane graphs $G_1$ and $G_2$ obtained by
  adding a vertex where any two edges cross in the embedding
  $\phi_1$, respectively, $\phi_2$ cross.  The fact that
  $\langle\varphi_1(1),\ldots,\varphi_1(n)\rangle$ and
  $\langle\varphi_2(1),\ldots,\varphi_2(n)\rangle$ have the same
  super-order type implies that $G_1$ and $G_2$ are two combinatorially
  equivalent drawings of the same planar graph.  Without loss of
  generality, we can assume that each obstacle $X_1\in S_1$ is an open
  polygon whose boundary is the face, $f_1$, of $G_1$ that contains $X_1$.
  For each such obstacle, we create an obstacle, $X_2\in S_2$, whose boundary
  is the face $f_2$ of $G_2$ that corresponds to $f_1$.  

  All that remains is to verify that an edge $uw$ is in $G$ if and only
  if the segment with endpoints $\varphi_2(u)$ and $\varphi_2(w)$ does not
  intersect any obstacle in $S_2$.
  By construction, no edge of the embedding $\varphi_2$ of $G$
  intersects any obstacle in $S_2$.  Because $P_1$ and $P_2$ have the
  same super-order type it is easy to verify that, for every non-edge
  $uw$ of $G$, the segment $\varphi_2(u)\varphi_2(w)$ intersects some
  obstacle in $S_2$. (Indeed, $\varphi_2(u)\varphi_2(w)$ intersects
  the obstacles of $S_2$ corresponding to those in $S_1$ intersected
  by $\varphi_1(u)\varphi_1(w)$.)  That is, $(\varphi_2,S_2)$ is an
  obstacle representation of $G$ using $h=|S_1|$ obstacles, as required.
\end{proof}

The next lemma shows that we can restrict our attention to embeddings
onto simple point sets.

\begin{lem}\lemlabel{simple}
  If a graph $G$ with vertex set $V=\{1,\ldots,n\}$ has an
  $h$-obstacle representation $(\varphi, S)$, then $G$ has an
  $h$-obstacle representation $(\varphi',S')$ in which the $\langle
  \varphi(1),\ldots,\varphi(n)\rangle$ is a simple point sequence.
\end{lem}

\begin{proof}
  We first note that, by rotation, we can assume that no two points
  in the image of $\varphi$ have the same x-coordinate.  Therefore,
  any degenerate sextuple in the image of $\varphi$ is not the result
  of a vertical line through two points in the sextuple.  Instead, each
  degenerate sextuple is the result of two parallel lines or of three
  lines passing through a common point.

  By growing the obstacles in $S$ into maximal open sets and then
  shrinking them slightly, we may assume that each obstacle in $S$ is
  an open set that is at some positive distance $\epsilon >0$ from each
  line segment $\varphi(u)\varphi(w)$ joining the two endpoints of each
  edge $uw$ in $G$.  Call this new set of obstacles $S'$.  We say that
  two points $a$ and $b$ are \emph{visible} if the open line segment
  with endpoints $a$ and $b$ does not intersect any obstacle in $S'$,
  otherwise we say that $a$ and $b$ are \emph{invisible}.

  If the image of $\varphi$ is a non-simple point set, then some
  point $a_1=\varphi(u)$ is involved in a degenerate sextuple
  $T=(a_1,a_2,b_1,b_2,c_1,c_2)$.  Then there exists a sufficiently
  small perturbation of $a_1$ that moves it to a new location $a_1'$ that
  simultaneously
  \begin{enumerate}
    \item eliminates all the degenerate sextuples that include $a_1$;
    \item does not change the type of any non-degenerate sextuple;
    \item does not result in any point $b=\varphi(w)$ that is visible to $a_1$
      being invisible to $a_1'$; and
    \item does not result in any point $b=\varphi(w)$ that is invisible to $a_1$
      being visible to $a_1'$.
  \end{enumerate}
  Note that the first two properties ensure that, by moving $a_1$ to $a_1'$,
  the number of degenerate sextuples decreases.  The last three properties
  ensure that the resulting embedding along with the obstacle set $S'$
  is an obstacle representation of $G$.  We can easily verify that such
  a point $a_1'$ exists because
  \begin{enumerate}
    \item For each degenerate sextuple that includes $a_1$ there are only
    a constant number of directions $(a-a')/\|a-a'\|$ that preserve the
    degeneracy of that sextuple.
    \item Changing the type of a non-denerate sextuple involving $a_1$
    requires moving $a_1$ by some distance $\delta >0$;  we can ensure
    that our perturbation moves $a_1$ by less than $\delta$.
    \item All obstacles are at distance $\epsilon >0$ from the
    edges of the embedding.  We can ensure that the perturbation moves
    $a_1$ by less than $\epsilon$.
    \item All obstacles are open sets, so every non-edge intersects the
    interior of one or more obstacles.  By making the perturbation of
    $a_1$ sufficiently small, each such non-edge continues to intersect
    the interiors of the same obstacles.
  \end{enumerate}
  The preceding perturbation step can be repeated until no degenerate
  sextuples remain to obtain the desired $h$-obstacle representation
  $(\varphi',S')$.
\end{proof}

What remains is to show that $N$, the number of super-order types
corresponding to point sets of size $n$ is not too big.  Luckily,
the methods used by Goodman and Pollack \cite{goodman.pollack:upper}
to upper bound the numbers of order types and combinatorial types
generalize readily to super-order types.  The proof of the following
result is given in \appref{order-type}.

\begin{lem}\lemlabel{order-type-count}
  The number of sequences in $\{-1,{+1}\}^{r}$ that are the super-order
  type of some simple point sequence of length $n$ is $e^{O(n\log n)}$.
\end{lem}

\subsection{Proof of \thmref{main}}

For completeness, we now spell out the proof of \thmref{main}.

\begin{proof}[Proof of \thmref{main}]
  Let $G$ be an Erd\"os-R\'enyi random graph with $n$ vertices.
  We say that the (point set which is the) image of $\varphi$ in an
  obstacle representation $(\varphi, S)$ \emph{supports} the obstacle
  representation.  Fix some simple super-order type on $n$ points.
  By \lemref{order-type}, all point sets with this super-order type
  support an obstacle representation of $G$ with $o(n/(\log\log n)^2)$
  obstacles or none of them do.  By \lemref{fixed}, the probability that
  all of them do is at most $p\le e^{-cn\log n}$ for every constant $c>0$.
  By the union bound and \lemref{order-type-count} the probability
  that there is any simple super-order type---and therefore any simple
  point set---that supports an obstacle representation of $G$ with
  $o(n/(\log\log n)^2)$ obstacles is
  \[
     \hat p = p\cdot e^{O(n\log n)} = e^{-cn\log n}\cdot e^{O(n\log n)} < 1
  \]
  for a sufficiently large constant $c$.  Therefore, with probability
  $1-\hat p > 0$, there is no simple point set that supports an obstacle
  representation of $G$ using $o(n/(\log\log n)^2)$ obstacles. We deduce
  that there exists some some graph, $G'$, with this property.  Finally,
  \lemref{simple} implies that we can ignore the restriction to simple
  point sets and deduce that $\obs(G')\in \Omega(n/(\log\log n)^2)$.
\end{proof}

\section{Remarks}

Our proof of \thmref{main} relates the problem of counting the number
of $n$-vertex graphs with obstacle number at most $h$ to the problem of
determining the worst-case obstacle number of a graph with $n$ vertices.
Currently, we use Mukkamala \etal's \thmref{pach-gang}, which proves an
upper-bound of $e^{O(hn\log^2 n)}$ on the number of $n$ vertex graphs
with obstacle number at most $h$.  Interestingly, their argument is an
encoding argument, which shows that any such graph can be encoded as the
order type of a set of $O(hn\log n)$ points followed by a list of the
points in this set that make up the vertices of the (polygonal) obstacles.
Their argument needs only order types (as opposed to super-order types)
since the point set that they specify includes the vertices of the
obstacles.

Any improvement on the upper-bound for the counting problem will
immediately translate into an improved lower-bound on the worst-case
obstacle number.  In particular, let $f(h,n)$ denote the number of
$n$-vertex graphs with obstacle number at most $h$ and let $\hat h(n)
= \max\{h:f(h,n) \le 2^{n^2/4}\}$.  Then our proof technique shows
that there exist graphs with obstacle number at least $n\hat{h}(c\log
n)/(c\log n)$. (\thmref{pach-gang} shows that $\hat{h}(c\log n)\in\Omega(\log n/(\log\log n)^2)$.)

We note that our technique gives an improved lower bound until someone is
able to prove that $\hat h(n)\in\Omega(n)$.  At this point, a simple
argument (see \cite[Theorem~3]{mukkamala.pach.ea:graphs}) shows that
there exists graphs with obstacle number at least $\hat{h}(n)$.

We conjecture that improved upper-bounds on $f(h,n)$ that reduce the
dependence on $h$ are the way forward:
\begin{conj}\conjlabel{h}
  $f(h,n) \le 2^{g(n)\cdot o(h)}$, where $g(n)\in O(n\log^2 n)$.
\end{conj}
In support of this conjecture, we observe that an upper bound of the
form $f(1,n)\le 2^{g(n)}$ is sufficient to give the crude upper bound
$f(h,n)\le 2^{h\cdot g(n)}$ since any graph with an $h$-obstacle
representation is the common intersection of $h$ graphs that each
have a 1-obstacle representation.  That is, if $\obs(G)\le h$, then
there exists $E_1,\ldots,E_h$ such that $G=(V,\bigcap_{i=1}^h E_i)$
and $\obs(V,E_i)=1$ for all $i\in \{1,\ldots,h\}$.  This seems like a
very crude upper bound in which many graphs are counted multiple times.
\conjref{h} asserts that this argument overestimates the dependence on $h$.

\section*{Acknowledgement}

This work was initiated at the \emph{Workshop on Geometry and Graphs},
held at the Bellairs Research Institute, March 10-15, 2013.  We are
grateful to the other workshop participants for providing a stimulating
research environment.

\bibliographystyle{abbrv}
\bibliography{obstacle}

\begin{thebibliography}{10}

\bibitem{alon.spencer:probabilistic}
N.~Alon and J.~H. Spencer.
\newblock {\em The Probabilistic Method}.
\newblock John Wiley {\&} Sons, Hoboken, third edition, 2008.

\bibitem{alpert.koch.ea:obstacle}
H.~Alpert, C.~Koch, and J.~D. Laison.
\newblock Obstacle numbers of graphs.
\newblock {\em Discrete {\&} Computational Geometry}, 44(1):223--244, 2010.

\bibitem{bland.vergnas:orientability}
R.~G. Bland and M.~L. Vergnas.
\newblock Orientability of matroids.
\newblock {\em J. Comb. Theory, Ser. B}, 24(1):94--123, 1978.

\bibitem{erdos.renyi:random}
P.~Erd\"os and A.~R\'enyi.
\newblock On random graphs.
\newblock {\em Publicationes Mathematicae}, 6:290--297, 1959.

\bibitem{folkman.lawrence:oriented}
J.~Folkman and J.~Lawrence.
\newblock Oriented matroids.
\newblock {\em J. Combin. Theory Ser. B}, 25:199--236, 1978.

\bibitem{fulek.saeedi:convex}
R.~Fulek, N.~Saeedi, and D.~Sari{\"o}z.
\newblock Convex obstacle numbers of outerplanar graphs and bipartite
  permutation graphs.
\newblock {\em CoRR}, abs/1104.4656, 2011.

\bibitem{goodman:on}
J.~E. Goodman.
\newblock On the combinatorial classification of nondegenerate configurations
  in the plane.
\newblock {\em J. Comb. Theory, Ser. A}, 29(2):220--235, 1980.

\bibitem{goodman.pollack:multidimensional}
J.~E. Goodman and R.~Pollack.
\newblock Multidimensional sorting.
\newblock {\em SIAM J. Comput.}, 12(3):484--507, 1983.

\bibitem{goodman.pollack:semispaces}
J.~E. Goodman and R.~Pollack.
\newblock Semispaces of configurations, cell complexes of arrangements.
\newblock {\em J. Comb. Theory, Ser. A}, 37(3):257--293, 1984.

\bibitem{goodman.pollack:upper}
J.~E. Goodman and R.~Pollack.
\newblock Upper bounds for configurations and polytopes in r$^{\mbox{d}}$.
\newblock {\em Discrete {\&} Computational Geometry}, 1:219--227, 1986.

\bibitem{johnson.sarioz:computing}
M.~P. Johnson and D.~Sari{\"o}z.
\newblock Computing the obstacle number of a plane graph.
\newblock {\em CoRR}, abs/1107.4624, 2011.

\bibitem{mukkamala.pach.ea:lower}
P.~Mukkamala, J.~Pach, and D.~P{\'a}lv{\"o}lgyi.
\newblock Lower bounds on the obstacle number of graphs.
\newblock {\em Electr. J. Comb.}, 19(2):P32, 2012.

\bibitem{mukkamala.pach.ea:graphs}
P.~Mukkamala, J.~Pach, and D.~Sari{\"o}z.
\newblock Graphs with large obstacle numbers.
\newblock In D.~M. Thilikos, editor, {\em WG}, volume 6410 of {\em Lecture
  Notes in Computer Science}, pages 292--303, 2010.

\bibitem{pach.sarioz:small}
J.~Pach and D.~Sari{\"o}z.
\newblock Small $(2,s)$-colorable graphs without 1-obstacle representations.
\newblock {\em CoRR}, abs/1012.5907, 2010.

\bibitem{pach.sarioz:on}
J.~Pach and D.~Sari{\"o}z.
\newblock On the structure of graphs with low obstacle number.
\newblock {\em Graphs and Combinatorics}, 27(3):465--473, 2011.

\bibitem{sarioz:approximating}
D.~Sari{\"o}z.
\newblock Approximating the obstacle number for a graph drawing efficiently.
\newblock In {\em CCCG}, 2011.

\end{thebibliography}

\appendix

\section{Proof of \lemref{order-type-count}}
\applabel{order-type}

\begin{proof}[Proof of \lemref{order-type-count}]
   For a point, $p$, let $\x(p)$ and $\y(p)$ denote the x- and
   y-coordinate, respectively, of $p$.  Consider what it means for a
   sextuple $T=(a_1,a_2,b_1,b_2,c_1,c_2)$, which defines three lines
   $A$, $B$, and $C$, to be degenerate.  This can occur, for example,
   if $A$ and $B$ are parallel.  The lines $A$ and $B$ are parallel if
   and only if
   \[
      \x(a_1-a_2)\cdot \y(b_1-b_2) - 
       \x(b_1-b_2)\cdot \y(a_1-a_2) = 0 \enspace .
   \]
   (This formula is a formalization of the less precise statement
   ``the slopes of $A$ and $B$ are the same.'')
   Observe that the preceding equation is a polynomial in
   $a_1,a_2,b_1,b_2$ of degree 2.

   Next, consider the case where $T$ is degenerate because $A$, $B$, and
   $C$ intersect in a common point or because one of $A$, $B$, or $C$
   is vertical.  This occurs if and only if the following determinant
   is undefined or equal to zero:
   \begin{equation}
   \left|\begin{array}{lll}
   \y(a_1)-\x(a_1)\left(\frac{\y(a_1-a2)}{\x(a_1-a_2)}\right) & \frac{\y(a_1)-\y(a_2)}{\x(a_1)-\x(a_2)} & 1 \\
   \y(b_1)-\x(b_1)\left(\frac{\y(b_1-b2)}{\x(b_1-b_2)}\right) & \frac{\y(b_1)-\y(b_2)}{\x(b_1)-\x(b_2)}  & 1 \\
   \y(c_1)-\x(c_1)\left(\frac{\y(c_1-c2)}{\x(c_1-c_2)}\right) & \frac{\y(c_1)-\y(c_2)}{\x(c_1)-\x(c_2)}  & 1 \\
   \end{array}\right| \enspace .
   \eqlabel{determinant}
   \end{equation}
   (The values in this matrix are the $\y$-intercepts and slopes of the
   lines $A$, $B$, and $C$.)
   Multiplying the matrix in \eqref{determinant} by 
   \[
      \x(a_1-a_2)\cdot
      \x(b_1-b_2)\cdot
      \x(c_1-c_2)
   \]
   yields a polynomial of degree $6$ in the six variables
   $a_1,a_2,b_1,b_2,c_1,c_2$ that is equal to zero if and only if $A$,
   $B$, and $C$ contain a common point or one of $A$, $B$, or $C$
   is vertical. (Recall that $\det(cA)=c^r\cdot\det(A)$ when $A$ is a
   $r\times r$ matrix.)

   For the remainder of the proof, we proceed exactly as in
   \cite{goodman.pollack:upper}.  We can treat any sequence of $n$
   points in $\R^2$ as a single point in $\R^{2n}$.  Applying the
   preceding conditions for determining the degeneracy to each of
   the $O(n^6)$ admissible sextuples of points results in a set of
   $O(n^6)$ polynomials in $2n$ variables, each of constant degree.
   By multiplying these polynomial together, we obtain a single
   polynomial, $P^*$, in $2n$ variables and having degree $d\in O(n^6)$.
   If $X\subset\R^{2n}$ is the zero set of $P^*$, then $\R^{2n}\setminus
   X$ has at most $(2+d)^{2n}=e^{O(n\log n)}$ connected components
   \cite[Lemma~2]{goodman.pollack:upper}.

   Observe that $X$ corresponds exactly to the set of non-simple
   point sequences and observe that a sextuple of points can not be moved
   continuously so that its type goes from $-1$ to $+1$, or vice versa,
   without its type becoming $0$ at some point during the movement.
   In particular, it is not possible to change the super-order type of
   a simple point sequence without going through a non-simple super-order type.
   Thus, within each component, $C$, of $\R^{2n}\setminus X$, the
   super-order type corresponding to every point in $C$ is the same.
   We conclude that the number of super-order types of simple point
   sequences is at most the number of components of $\R^{2n}\setminus X$,
   which is $e^{O(n\log n)}$, as required.
\end{proof}

\section*{Authors}

\noindent
\includegraphics[width=.45\textwidth]{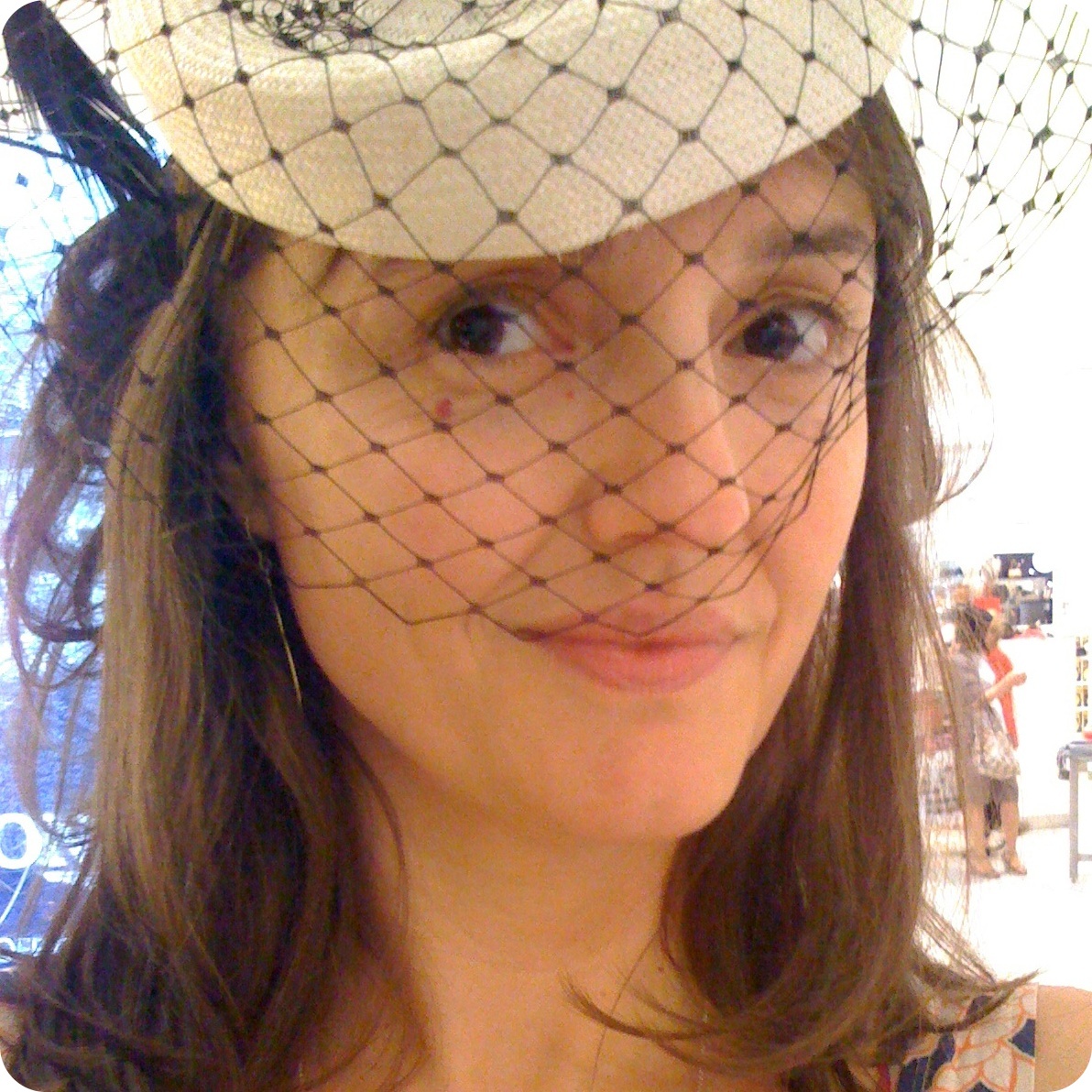}%
\hspace{.1\textwidth}%
\includegraphics[width=.45\textwidth]{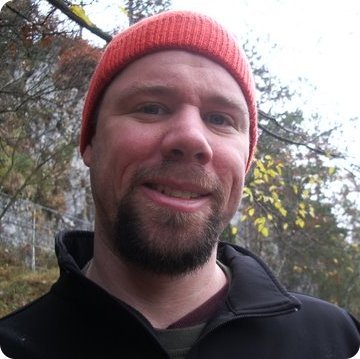}%

\noindent\emph{Vida Dujmovi\'c.}
School of Mathematics and Statistics and Department of Systems and Computer Engineering, Carleton University

\noindent\emph{Pat Morin.}
School of Computer Scence, Carleton University

\end{document}